\documentclass{article}

\usepackage{graphicx}
\usepackage{indentfirst}
\usepackage{amsmath,amsfonts,amsthm,amssymb,amsrefs}
\usepackage{mathrsfs,stmaryrd}
\usepackage{amscd}
\usepackage{hyperref}

\usepackage{tikz}
\usetikzlibrary{knots}

\usepackage{relsize}

\DeclareMathAlphabet{\mathsfsl}{OT1}{cmss}{m}{sl}

\newtheorem{thm}{Theorem}[section]
\newtheorem{lem}[thm]{Lemma}

\newtheorem{prop}[thm]{Proposition}
\newtheorem{conj}[thm]{Conjecture}

\theoremstyle{definition}
\newtheorem{defn}[thm]{Definition}

\begin{document}

\title{A characterization of $T_{2g+1,2}$ among alternating knots}

\author{{\Large Yi NI}\\{\normalsize Department of Mathematics, Caltech, MC 253-37}\\
{\normalsize 1200 E California Blvd, Pasadena, CA
91125}\\{\small\it Emai\/l\/:\quad\rm yini@caltech.edu}}

\date{}
\maketitle

\begin{center}
{\it Dedicated to the memory of Professor John Horton Conway}
\end{center}

\ 

\begin{abstract}
Let $K$ be a genus $g$ alternating knot with Alexander polynomial $\Delta_K(T)=\sum_{i=-g}^ga_iT^i$. We show that if $|a_g|=|a_{g-1}|$, then $K$ is the torus knot $T_{2g+1,\pm2}$. This is a special case of the Fox Trapezoidal Conjecture.
The proof uses Ozsv\'ath and Szab\'o's work on alternating knots.
\end{abstract}

\section{Introduction}

Alternating knots have many good properties. For example, the information from the Alexander polynomial of an alternating knot $K$ determines the genus of $K$ and whether $K$ is fibered \cite{Crowell,Murasugi}. Even so, there are still some open problems about alternating knots. One of these problems is the following conjecture made by Fox \cite[Problem~12]{Fox}.

\begin{conj}[Fox Trapezoidal Conjecture]\label{conj:Fox}
Let $K$ be an alternating knot with normalized Alexander polynomial 
\begin{equation}\label{eq:Alexander}
\Delta_K(T)=\sum_{i=-g}^ga_iT^i,
\end{equation}
where $g$ is the genus of $K$.
 Then \[|a_i|\le|a_{i-1}| \text{ when }0< i\le g.\] Moreover, if $|a_i|=|a_{i-1}|$ for some $i$, then $|a_j|=|a_i|$ whenever $0\le j\le i$.
\end{conj}

This conjecture was known for $2$--bridge knots \cite{Hartley} and alternating arborescent knots \cite{MurasugiAlgebraic}. Using Heegaard Floer homology, Ozsv\'ath and Szab\'o \cite{OSzAlternating} proved the first part of the conjecture for $i=g$. See (\ref{eq:agBound}) for the precise inequality. As a result, they proved the conjecture for genus--$2$ knots. 

In this paper, we will prove the second part of Conjecture~\ref{conj:Fox} for $i=g$. In this case, we will get a stronger conclusion.

\begin{thm}\label{thm:Alternating}
Let $K$ be an alternating knot with normalized Alexander polynomial given by (\ref{eq:Alexander}), where $g$ is the genus of $K$. If $|a_g|=|a_{g-1}|$, then $K$ or its mirror is the torus knot $T_{2g+1,2}$.
\end{thm}

Our proof uses Ozsv\'ath and Szab\'o's work \cite{OSzAlternating}.

This paper is organized as follows.  In Section~\ref{sect:Thin}, we prove that if a knot $K$ has thin knot Floer homology, and $|a_g|=|a_{g-1}|$, then $K$ is a strongly quasipositive fibered knot.  In Section~\ref{sect:SQPFalt}, we prove that strongly quasipositive fibered alternating knots are connected sums of torus knots of the form $T_{2n+1,2}$. Hence we get a proof of Theorem~\ref{thm:Alternating}. 

\vspace{5pt}\noindent{\bf Acknowledgements.}\quad  The author was
partially supported by NSF grant number DMS-1811900. 


\section{Thin knots with $|a_g|=|a_{g-1}|$}\label{sect:Thin}

Let $K\subset S^3$ be a knot with knot Floer homology \cite{OSzKnot,RasThesis}
\[
\widehat{HFK}(S^3,K)=\bigoplus_{i,j\in\mathbb Z^2}\widehat{HFK}_j(S^3,K,i).
\]
We say the knot Floer homology is {\it thin}, if it is supported in the line 
\[
j=i-\tau,
\]
where $\tau=\tau(K)$ is the concordance invariant defined in \cite{OSz4Genus}.

By work of Hedden \cite{Hedden}, we will make the following definition of strongly quasipositive fibered knots. We do not need the original definition of strong quasipositivity in \cite{Rudolph}.

\begin{defn}
A {\it strongly quasipositive} fibered knot is a fibered knot $K\subset S^3$, such that the open book with binding $K$ supports the tight contact structure on $S^3$.
\end{defn}

Now we can state the main result we will prove in this section.

\begin{prop}\label{prop:Thin}
Let $K\subset S^3$ be a knot with thin knot Floer homology. Let the normalized Alexander polynomial be given by (\ref{eq:Alexander}). If $|a_g|=|a_{g-1}|$, then $K$ or its mirror is a strongly quasipositive fibered knot.
\end{prop}

Let $S_0^3(K)$ be the manifold obtained by $0$--surgery on $K$.
Ozsv\'ath and Szab\'o proved that if $\widehat{HFK}(S^3,K)$ is thin and $\tau(K)\ge0$, then
\begin{equation}\label{eq:HF+}
HF^+(S^3_0(K),s)\cong \mathbb Z^{b_s}\oplus(\mathbb Z[U]/U^{\delta(-2\tau,s)})
\end{equation}
for $s>0$,
where 
\[
\delta(-2\tau,s)=\max\{0,\lceil\frac{|\tau|-|s|}2\rceil\}
\]
and 
\[
(-1)^{s-\tau}b_s=\delta(-2\tau,s)-t_s(K)
\]
with 
\[
t_s(K)=\sum_{j=1}^{\infty}ja_{s+j}.
\]
See \cite[Theorem~1.4]{OSzAlternating} and the paragraph after it.

Using (\ref{eq:HF+}), one can deduce the following inequality as in \cite{OSzAlternating}:
\begin{equation}\label{eq:agBound}
|a_{g-1}|\ge2|a_g|+\left\{
\begin{array}{ll}
-1 &\text{if }|\tau|=g\\
1 &\text{if }|\tau|=g-1\\
0 &\text{otherwise.}
\end{array}
\right.
\end{equation}

\begin{proof}[Proof of Proposition~\ref{prop:Thin}]
It follows from \cite{OSzGenus} that $a_g\ne0$.
If $|a_g|=|a_{g-1}|$, then by (\ref{eq:agBound}) we must have
\[
|a_g|=1, |\tau|=g.
\]
By \cite{Gh,NiFibred}, $K$ is fibered. Replacing $K$ with its mirror if necessary, we may assume $\tau=g$. It follows from \cite[Corollary~1.7]{OSzAlternating} that the open book with binding $K$ supports the tight contact structure.
\end{proof}


\section{Strongly quasipositive fibered alternating knots}\label{sect:SQPFalt}

Suppose that $K$ is a fibered alternating link.
Let $\mathcal D\subset S^2$ be a reduced connected alternating diagram of $K$. Applying Seifert's algorithm to $\mathcal D$, we can get a Seifert surface $F$ which is a union of disks  and twisted bands corresponding to the crossings in $\mathcal D$. We call the disks {\it Seifert disks} with boundary {\it Seifert circles}, and call the twisted bands {\it Seifert bands}.
By \cite[Theorem~5.1]{GabaiFibred}, $F$ is a fiber of the fibration of $S^3\setminus K$ over $S^1$.  

Following \cite{GHY}, we say a Seifert circle is {\it nested}, if each of its complementary regions contains another Seifert
circle. It is well-known that $F$ decomposes as a Murasugi sum of two surfaces along a nested Seifert circle $C$ \cite{Murasugi,Stallings}. 
More precisely, let $D_1,D_2$ be the two disks bounded by $C$. Let $\mathcal B_i$ be the union of Seifert bands connecting $C$ to Seifert circles in $D_i$, $i=1,2$. We cut $F$ open along $\mathcal B_{3-i}\cap C$ to get a disconnected surface. Let $F_i$ be the component such that the projection of $\partial F_i$ is supported in $D_i$. Then $F$ is a Murasugi sum of $F_1$ and $F_2$. Gabai \cite{GabaiMurasugi} proved that $F$ is a fiber of a fibration of $S^3\setminus K$ if and only each $F_i$ is a fiber of a fibration of $S^3\setminus \partial F_i$, $i=1,2$.

\begin{defn}
If a diagram contains no nested Seifert circles, then this diagram is {\it special} as defined in \cite{Murasugi}.
\end{defn}

Suppose that $\mathcal D\subset S^2$ is a reduced connected special alternating diagram for a link $K$. Let $S_1,\dots,S_k$ be the Seifert circles in $\mathcal D$. Since $\mathcal D$ is special, these Seifert circles bound disjoint disks $D_1,\dots,D_k$.
We color the complementary regions of $\mathcal D$ by two colors black and white, so that two regions sharing an edge have different colors. The coloring convention is that the disks $D_1,\dots,D_k$ have the black color. Clearly, there are no other black regions.
We will construct the black graph $\Gamma_B$ and the white graph $\Gamma_W$ as usual. Namely, the vertices in $\Gamma_B$ (or $\Gamma_W$) are the black (or white) regions, and the edges correspond to the crossings. These two graphs are embedded in $S^2$ as a pair of dual graphs.
We also construct the reduced black graph $\Gamma_B^r$ by deleting all but one edges connecting two vertices $v_i$ and $v_j$ if there is any edge connecting them.

The following proposition can be found in \cite[Propositions~13.24 and~13.25]{BZ}.

\begin{prop}\label{prop:Graph}
Suppose that $\mathcal D\subset S^2$ is a reduced connected special alternating diagram for a fibered link $K$, then all but one vertices in $\Gamma_W$ have valence $2$. As a result, $K$ is a connected sum of torus links 
\[K=\#_{i=1}^{\ell}T_{k_i,2}.\]
\end{prop}

From Proposition~\ref{prop:Graph}, it is not hard to get the following characterization of $\mathcal D$ in terms of $\Gamma_B^r$.

\begin{lem}\label{lem:Tree}
Under the same assumptions as in Proposition~\ref{prop:Graph}, the graph $\Gamma_B^r$ is a tree.
\end{lem}
\begin{proof}
Since $D$ is connected, $\Gamma_B^r$ is also connected.
If $\Gamma_B^r$ contains only two vertices, there is exactly one edge by the definition of $\Gamma_B^r$, so our conclusion holds. From now on, we assume $\Gamma_B^r$  has at least three vertices.
 Let $R$ be a complementary region of $\Gamma_B^r$, then it is not a bigon since any two vertices in $\Gamma_B^r$ are connected by at most one edge and $\Gamma_B^r$  has at least three vertices.
Let $v$ be the vertex corresponding to $R$ in $\Gamma_W$, then $v$ has valence $>2$. By Proposition~\ref{prop:Graph}, $\Gamma_B^r$ has at most one complementary region, which means that $\Gamma_B^r$ is a tree.
\end{proof}

\begin{lem}\label{lem:TwoEdges}
Under the same assumptions as in Proposition~\ref{prop:Graph}, if two vertices in $\Gamma_B$ are connected by an edge, then they are connected by at least two edges.
\end{lem}
\begin{proof}
Using Lemma~\ref{lem:Tree}, if $D_i$ and $D_j$ are connected through only one crossing, then $\mathcal D$ is not reduced, a contradiction.
\end{proof}

We say two Seifert bands are {\it parallel} if they connect the same
two Seifert disks. The following lemma is well-known. See, for example, \cite[Proposition~5.1]{GHY}.

\begin{lem}\label{lem:Deplumbing}
If two Seifert bands are parallel, then we can deplumb a Hopf band from $F$. The resulting surface can be obtained by removing one of the bands from $F$.
\end{lem}


\begin{lem}\label{lem:NoParallel}
Let $K$ be a strongly quasipositive fibered alternating knot, and let $\mathcal D$ be a reduced connected alternating diagram for $K$.
Let $C$ be a nested Seifert circle. If $C$ is connected to two pairs of parallel bands, then these two pairs of bands are on the same side of $C$.
\end{lem}
\begin{proof}
If $C$ is connected to two pairs of parallel bands on different sides of $C$, then we can deplumb a negative Hopf band from $F$. See Figure~\ref{fig:TwoSides}. Hence the open book with page $F$ supports an overtwisted contact structure \cite[Lemma~4.1]{Goodman}, a contradiction.
\end{proof}

\begin{figure}[ht]
\begin{center}
\scalebox{0.5}{
\begin{tikzpicture}
\begin{knot}[clip width=6, consider self intersections,]
\strand
(0,2)  ..  controls  (0.2,4) and (0.3,5) ..  
(1,5)  ..  controls  +(1,0) and +(-1,0)  ..
(3,6)  ..  controls  +(1,0) and +(-1,0)  ..
(5,5)  ..  controls  +(1,0) and +(-1,0)  ..
(7,6)  ..  controls  (8,6.2) and (8,6.8)   ..
(7,7)  ..  controls  (5,7.5) and (2,7.5) ..
(1,7)  ..  controls  (0,6.8) and (0,6.2) ..
(1,6)  ..  controls  +(1,0) and +(-1,0)  ..
(3,5)  ..  controls  +(1,0) and +(-1,0)  ..
(5,6)  ..  controls  +(1,0) and +(-1,0)  ..
(7,5)  ..  controls  (8,4.5) and (8,1.5)  ..
(7,1)  ..  controls  +(-1,0) and +(1,0) ..
(5,2)  ..  controls  +(-1,0) and +(1,0) ..
(3,1)  ..  controls  +(-1,0) and +(1,0) ..
(1,2)  ..  controls  (0.5,2.2) and (0.5,2.8) ..
(1,3)  ..  controls  (2,3.5) and (5,3.5) ..
(7,3)  ..  controls  (7.5,2.8) and (7.5,2.2) ..
(7,2)  ..  controls  +(-1,0) and +(1,0) ..
(5,1)  ..  controls  +(-1,0) and +(1,0) ..
(3,2)  ..  controls  +(-1,0) and +(1,0) ..
(1,1)  ..  controls  (0.2,1.2) ..
(0,2) ;
\flipcrossings{2,5};
\draw[red, dashed] (3,1.5) circle [radius=1];
\draw[blue, dashed] (3,5.5) circle [radius=1];
\end{knot}
\end{tikzpicture}
}
\end{center}
\caption{\label{fig:TwoSides}If two collections of parallel bands are on different sides of a nested Seifert circle, we can deplumb a positive Hopf band and a negative Hopf band. 
The two dashed circles are the cores of the Hopf bands.}
\end{figure}
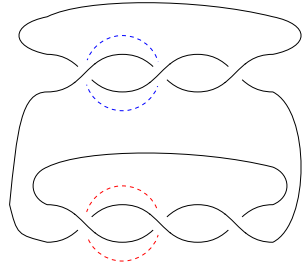

\begin{prop}\label{prop:SQP}
Let $K$ be a strongly quasipositive fibered alternating knot. Then $K$ is a connected sum of torus knots of the form $T_{2n+1,2}$ for $n>0$.
\end{prop}
\begin{proof}
If $\mathcal D$ is special, by Proposition~\ref{prop:Graph}, $K$ is a connected sum of torus knots $T_{2n_i+1,2}$. Since $K$ is strongly quasipositive, each $n_i$ must be positive, so our conclusion holds.

Now we assume that $\mathcal D$ contains at least one nested Seifert circle.
We say a nested Seifert circle is extremal, if one of its complementary regions contains no other nested Seifert circles. 
Let $C_1,\dots,C_m$ be a maximal collection of extremal nested Seifert circles in $\mathcal D$, and let $R_i$ be the complementary region of $C_i$ which contains no other nested Seifert circles. Then $R_1,\dots,R_m$ are mutually disjoint. Let $\mathcal D'$ be the diagram obtained from $\mathcal D$ by Murasugi desumming along $C_1\cup\cdots\cup C_m$. Let $\mathcal D_i$ be the part of $\mathcal D'$ supported in $R_i$, and let \[\mathcal D^*=\mathcal D'\setminus(\cup_{i=1}^m\mathcal D_i).\]
By \cite{GabaiMurasugi}, $\mathcal D^*$ and $\mathcal D_i$ are alternating diagrams representing fibered links.

Since $R_i$ contains no other nested Seifert circles, $\mathcal D_i$ is special. By Lemma~\ref{lem:TwoEdges}, $C_i$ is connected to another circle in $R_i$ by at least a pair of parallel bands.

We claim that $\mathcal D^*$ is special.
Otherwise, let $C$ be an extremal nested Seifert circle, and let $R$ be the complementary region of $C$ which contains no other nested Seifert circles in $\mathcal D^*$. Since $C_1,\dots,C_m$ is a maximal collection of extremal nested Seifert circles, $R$ must contain at least one $C_i$. By Lemma~\ref{lem:TwoEdges}, $C_i$ is connected to another circle in $R\setminus R_i$ (including $C$) by at least a pair of parallel bands.
This is a contradiction to Lemma~\ref{lem:NoParallel}.

Now $\mathcal D^*$ is special. There are at least two Seifert circles in $\mathcal D^*$, since $C_1$ is nested in $\mathcal D$. By Lemma~\ref{lem:TwoEdges}, $C_1$ is connected to another Seifert circle in $\mathcal D^*$ by at least a pair of parallel bands. We again get a contradiction to Lemma~\ref{lem:NoParallel}. Hence $\mathcal D$ does not contain any nested Seifert circle. This finishes our proof.
\end{proof}

\begin{proof}[Proof of Theorem~\ref{thm:Alternating}]
By \cite{OSzAlternating}, $\widehat{HFK}(S^3,K)$ is thin. It follows from Proposition~\ref{prop:Thin} that $K$ is strong quasipositive and fibered. Using Proposition~\ref{prop:SQP}, $K$ is a connected sum of $T_{2n_i+1,2}$. The condition on the Alexander polynomial forces $K$ to be $T_{2g+1,2}$. 
\end{proof}



\end{document}